\documentclass[11pt,a4paper]{amsart}
\usepackage{amsmath,amssymb, amsbsy}
\usepackage{color,psfrag}
\usepackage[dvips]{graphicx}
\usepackage{color}
\usepackage{mathtools}
\usepackage{hyperref}
\theoremstyle{plain}

\newtheorem{theorem}{Theorem}[section]

\newtheorem{lemma}{Lemma}[section]
\newtheorem{corollary}{Corollary}[section]
\newtheorem{remark}{Remark}[section]
\newtheorem{definition}{Definition}[section]

\newcommand{\hn}{\mathbb{H}^{N}}

\newcommand{\dv}{\: {\rm d}v_{\hn}}

\newcommand{\dr}{\: {\rm d}r}

\numberwithin{equation}{section} \allowdisplaybreaks

\usepackage[text={6in,8.6in},centering]{geometry}
\begin{document}
	
	\title[Hardy-Rellich and Poincar\'e identities]{Hardy-Rellich and second order Poincar\'e identities\\ on the hyperbolic space via Bessel pairs} 

	\author[Elvise BERCHIO]{Elvise BERCHIO}
\address{\hbox{\parbox{5.7in}{\medskip\noindent{Dipartimento di Scienze Matematiche, \\
Politecnico di Torino,\\
       Corso Duca degli Abruzzi 24, 10129 Torino, Italy. \\[3pt]
       \em{E-mail address: }{\tt elvise.berchio@polito.it}}}}}

	\author[Debdip Ganguly]{Debdip Ganguly}
	\address{\hbox{\parbox{5.7in}{\medskip\noindent{Department of Mathematics,\\
					Indian Institute of Technology Delhi,\\
					IIT Campus, Hauz Khas, Delhi,\\
					New Delhi 110016, India. \\[3pt]
					\em{E-mail address: }{\tt 
						debdipmath@gmail.com}}}}}
					
	\author[Prasun Roychowdhury]{Prasun Roychowdhury}
	\address{\hbox{\parbox{5.7in}{\medskip\noindent{Department of Mathematics,\\
					Indian Institute of Science Education and Research,\\
					Dr. Homi Bhabha Road, Pashan,\\
					Pune 411008, India. \\[3pt]
					\em{E-mail address: }{\tt prasunroychowdhury1994@gmail.com}}}}}
	
	\date{\today}

	\maketitle
	
	\begin{abstract}
		We prove a family of Hardy-Rellich and Poincar\'e identities and inequalities on the hyperbolic space having, as particular cases, improved Hardy-Rellich, Rellich and second order Poincar\'e inequalities. All remainder terms provided considerably improve those already known in literature, and all identities hold with same constants for radial operators also. Furthermore, as applications of the main results, second order versions of the uncertainty principle on the hyperbolic space are derived.
	\end{abstract}
	

\section{Introduction}
Let $\hn$ with $N\geq 2$ denote the most important example of Cartan-Hadamard manifold, namely the hyperbolic space and let $\lambda_{1}(\hn)$ denote the bottom of the spectrum of $\Delta_{\hn}$ which is explicitly given by
\begin{equation}\label{poincare}
  \lambda_{1}(\hn)= \inf_{u \in C_{c}^{\infty}(\hn)  \setminus \{ 0 \}} \dfrac{\int_{\hn} |\nabla_{\hn} u|^2 \  \emph{\rm d}v_{\hn}}{\int_{\hn} u^2 \  \emph{\rm d}v_{\hn}}=\left(\frac{N-1}{2} \right)^2 \,.
  \end{equation}
 The present paper takes its origin from the following family of Hardy-Poincar\'e inequalities recently proved in \cite{BGGP}: for all $N-2 \leq  \lambda \leq  \lambda_{1}(\hn)$ and all $u \in C_{c}^{\infty} (\hn \setminus \{ x_0 \})$ there holds
\begin{equation}\label{improved-poinc-lambda}
\begin{aligned}
 &\int_{\hn}|\nabla_{\hn} u|^2 \emph{\rm d}v_{\hn} \geq \lambda \int_{\hn}  u^2 \emph{\rm d}v_{\hn}+ h_N^2(\lambda) \int_{\hn} \frac{u^2}{r^2} \emph{\rm d}v_{\hn} \\
 & +   \left[\left(\frac{N-2}{2} \right)^2-h_N^2(\lambda)\right] \int_{\hn}  \frac{u^2}{\sinh^2 r} \emph{\rm d}v_{\hn}+  \gamma_{N}(\lambda)\,h_N(\lambda) \int_{\hn} \frac{r \coth r - 1}{r^2}\, u^2\emph{\rm d}v_{\hn} \,
\end{aligned}
\end{equation}
where $\gamma_{N}(\lambda):=\sqrt{(N-1)^2-4\lambda}$, $h_N(\lambda):=\frac{\gamma_{N}(\lambda)+1}{2}$ and $r:={\rm d}(x, x_0)$ is the geodesic distance from a fixed pole $x_0 \in \hn$. We notice that the function $\frac{r \coth r - 1}{r^2}$ is positive while the map $[N-2,  \lambda_{1}(\hn)] \ni \lambda \mapsto h_N(\lambda)$ is decreasing. Furthermore, for $N\geq 3$, there holds $\frac{1}{4}\leq h_N^2(\lambda)\leq \left(\frac{N-2}{2} \right)^2$ and, for all $\lambda$, one locally recovers the optimal Hardy weight: $\left(\frac{N-2}{2} \right)^2 \frac{1}{r^2}$. Besides, denoted with $V_\lambda$ the positive potential at the r.h.s. of \eqref{improved-poinc-lambda}, the operator $ -\Delta_{\hn} -V_{\lambda}(r)$ is \emph{critical} in $\hn \setminus \{ x_0 \}$ in the sense that the inequality
$
\int_{\hn} |\nabla_{\hn} u|^2 \ \emph{d}v_{\hn} \geq \int_{\hn} V u^2 \ \emph{d}v_{\hn}$ is not valid for all $u \in C_{c}^{\infty} (\hn \setminus \{ x_0 \})$ if $V \gneqq V_{\lambda} $. 

The interest of \eqref{improved-poinc-lambda} relies on the fact that it provides in a single inequality, proved by means of a unified approach, an optimal improvement (in the sense of adding nonnegative terms in the right side of the inequality) of the Poincar\'e inequality \eqref{poincare} and an optimal improvement  of the Hardy inequality. Indeed, for $ \lambda =\lambda_{1}(\hn)$ ($\gamma_{N}=0$) inequality \eqref{improved-poinc-lambda} becomes the improved Poincar\'e inequality:
\begin{align}\label{poincareeq}
&\int_{\hn} |\nabla_{\hn} u|^2 \ {\rm d}v_{\hn}\geq  \left(\frac{N-1}{2} \right)^2 \int_{\hn} u^2 \ {\rm d}v_{\hn} \notag \\
&+ \frac{1}{4} \int_{\hn} \frac{u^2}{r^2} \ {\rm d}v_{\hn} + \frac{(N-1)(N-3)}{4} \int_{\hn} \frac{u^2}{\sinh^2 r}  \, {\rm d}v_{\hn}\,,
\end{align}

for all $u \in C_{c}^{\infty} (\hn \setminus \{ x_0 \})$ with $N\geq 2$. Instead, for $\lambda=N-2$ ($\gamma_{N}=N-3$) \eqref{improved-poinc-lambda} becomes the improved Hardy inequality:
\begin{align}\label{improved-hardy-eq}
&\int_{\hn} |\nabla_{\hn} u|^2 \ \emph{\rm d}v_{\hn}  \geq \left(\frac{N-2}{2} \right)^2 \int_{\hn} \frac{u^2}{r^2} \ \emph{\rm d}v_{\hn} \notag\\
 &+  (N-2) \int_{\hn} u^2 \ \emph{\rm d}v_{\hn}  + \frac{(N-2)(N-3)}{2} \int_{\hn}   \frac{r \coth r - 1}{r^2} \, u^2 \ \emph{\rm d}v_{\hn}\,,
\end{align}
for all $u \in C_{c}^{\infty} (\hn \setminus \{ x_0 \})$ with $N\geq 3$. As concerns inequality \eqref{poincareeq}, we recall that it has been shown first in \cite{AK} and then, with different methods, adapted to larger classes of manifolds in \cite{EDG} where criticality has also been shown. Very recently, a further development has been done in \cite{FLLM} where, by using the notion of Bessel pairs, it has been proved that a further positive term of the form $\int_{\hn}\frac{r}{\sinh^{N-1}r}\big|\nabla_{\hn}\big(u\frac{\sinh^{\frac{N-1}{2}}r }{r}\big)\big|^2  dv_{\hn}$ can be added at the r.h.s. of  \eqref{poincareeq} so that the inequality becomes an equality. Clearly, this is not in contrast with the criticality proved in \cite{EDG} since the added term is not of the form $V u^2$. We refer the interested reader to \cite{BAGG} for the $L^p$ version of \eqref{poincareeq}, and to \cite{BRM2} for remainder terms of \eqref{poincare} involving the Green's function of the Laplacian. 

Regarding \eqref{improved-hardy-eq}, it's worth recalling that generalizations to Riemannian manifolds of the classical euclidean Hardy inequality have been intensively pursued after the seminal work of Carron \cite{Carron}. In particular, on Cartan-Hadamard manifolds the optimal constant is known to be $\left(\frac{N-2}{2} \right)^2$ and improvements of the Hardy inequality have been given e.g., in \cite{Dambrosio,FLLM,Kombe1, Kombe2,AKR,YSK}. This is in contrast to what happens in the Euclidean setting where the operator $-\Delta_{\mathbb{R}^{N}} - \left(\frac{N-2}{2} \right)^2  \frac{1}{|x|^2} $ is known to be  \emph{critical} in $\mathbb{R}^{N} \setminus \{ 0 \}$ (see \cite{pinch}). In particular, in inequality \eqref{improved-hardy-eq} the effect of the curvature allows to provide a remainder term of $L^2$-type, therefore of the same kind of that given in the seminal paper by Brezis-Vazquez \cite{Brezis} for the Hardy inequality on euclidean \emph{bounded} domains. \par
\smallskip\par

The above mentioned results make it natural to investigate the existence of a family of inequalities extending \eqref{improved-poinc-lambda} to the second order, that is an inequality including either improvement of the second order Poincar\'e inequalities:
	\begin{align}\label{higher_order_poin}
	  \int_{\hn}(\Delta_{\hn} u)^2 {\rm d}v_{\hn}\geq \bigg(\frac{N-1}{2}\bigg)^{2(2-l)} \int_{\hn} |\nabla_{\hn}^l u|^2{\rm d}v_{\hn}\quad (l=0  \text{ or } l=1)
	\end{align}
	for all $u\in C_c^\infty(\hn)$ ($N\geq 2$), and improvement of the second order Hardy inequalities:
 \begin{align}\label{rellich-general}
	  \int_{\hn}(\Delta_{\hn} u)^2 {\rm d}v_{\hn}\geq \frac{N^2}{4}\,\bigg(\frac{N-4}{2}\bigg)^{2(1-l)} \int_{\hn} \frac{|\nabla_{\hn}^l u|^2}{r^{4-2l}}{\rm d}v_{\hn} \quad (l=0  \text{ or } l=1)
	\end{align}
 for all $u\in C_c^\infty(\hn)$ ($N\geq 5$), i.e. the Rellich inequality which comes for $l=0$ and the Hardy-Rellich inequality for $l=1$. We recall that inequalities \eqref{higher_order_poin} are known from \cite{ngo} and \cite{SK} with optimal constants, while improvements have been provided in \cite{EG}-\cite{EDG} and, for radial operators, in \cite{EGR}-\cite{prc-20}. Instead, inequalities \eqref{rellich-general} were firstly studied in \cite{Kombe1} and in \cite{YSK}, where the optimality of the constants was proved together with the existence of some remainder terms. More recently, a stronger version of \eqref{rellich-general}, only involving radial operators and still holding with same constants, has been obtained in \cite{VHN}. See also \cite{KR} for improved versions of \eqref{rellich-general} in the general framework of Finsler-Hadamard manifolds. 
 
 In the present paper we complete the picture of results in $\hn$ by proving a family of inequalities including either an improved version of \eqref{higher_order_poin} and an improved version of \eqref{rellich-general} when $l=1$, therefore extending \eqref{improved-poinc-lambda} to the second order, see Theorem  \ref{improved-hardy} below. Furthermore, in Theorem \ref{improved-hardy-radial}, we show that the obtained family of inequalities reads as a family of \emph{identities}  for radial operators (also for non radial functions) giving a more precise understanding of the remainder terms provided. A fine exploitation of these results also allows to obtain improved versions of \eqref{higher_order_poin} and of \eqref{rellich-general} for $l=0$ in such a way to exhaust the second order scenario, see Corollaries \ref{2} and \ref{4}. As far we are aware, all the improvements provided have a stronger positive impact, on the r.h.s. of \eqref{higher_order_poin} and of \eqref{rellich-general}, than those already known in literature, see Remark \ref{comparison} in the following.
 
 We notice that \eqref{improved-poinc-lambda} was proved by means of a unified approach based on criticality theory, well established for \emph{second} order operators only (see \cite{pinch}), together with the exploitation of a family of explicit radial solutions to the associated equations. Therefore, a similar approach seems not applicable in the higher order case. Here, drawing primary motivation from the seminal paper \cite{GM}, we extend \eqref{improved-poinc-lambda} to the second order by using the notion of  \emph{Bessel pair}. This notion has been very recently developed in \cite{FLLM} on Cartan-Hadamard manifolds to establish several interesting Hardy identities and inequalities which, in particular, generalise many well-known Hardy inequalities on Cartan-Hadamard manifolds. By combining some ideas from \cite{FLLM}-\cite{GM}, and through delicate computations with spherical harmonics, in the present article we develop the method of Bessel pairs to derive general abstract Rellich inequalities and identities on $\hn$ that we employ to prove our main results, i.e., Theorems \ref{improved-hardy-radial} and \ref{improved-hardy}. In this way, we get either Poincar\'e and Hardy-Rellich identities, and improved inequalities, by means of a unified proof where the key ingredient is the clever construction of a family of Bessel pairs, see \eqref{potential} in the following. Finally, as applications of the obtained inequalities, we derive quantitative versions of the second order Heisenberg-Pauli-Weyl uncertainty principle, see Section \ref{uncertainity}. As far as we know, the results provided represent the first examples of second order Heisenberg-Pauli-Weyl uncertainty principle in the Hyperbolic context.

  \medskip

The paper is organized as follows: in Section \ref{main} we introduce some of the notations and we state our main results, i.e. Poincar\'e and Hardy-Rellich identities and related improved inequalities; furthermore, in this section, we also state second order versions of the Heisenberg-Pauli-Weyl uncertainty principle. In Section \ref{bessel} we provide abstract Rellich identities and inequalities via Bessel pairs together with a related Heisenberg-Pauli-Weyl uncertainty principle. Section \ref{applications} is devoted to the proofs of the results stated in Section \ref{main} by exploiting the results stated in Section \ref{bessel}, while Section \ref{bessel proofs} contains the proofs of the results stated in Section \ref{bessel}. Finally, in the Appendix we present a family of improved Hardy-Poincar\'e identities which follows as a corollary from \cite[Theorem 3.2]{FLLM}, see Lemma \ref{th_Bessel_hardy} below, by exploiting the family of Bessel pairs introduced in Section \ref{applications}. In particular, these identities give a deeper understanding of \eqref{improved-poinc-lambda} and include \cite[Theorem~1.4]{FLLM} as a particular case.

\section{Main results}\label{main}

\subsection{Notations} From now onward, if nothing is specified, we will always assume $N\geq 2$.
It is well known that the $N$-dimensional hyperbolic space $\hn$ admits a polar coordinate decomposition structure. Namely, for $x\in\hn$ we can write $x=(r, \Theta)=(r, \theta_{1},\ldots, \theta_{N-1})\in(0,\infty)\times\mathbb{S}^{N-1}$, where $r$ denotes the geodesic distance between the point $x$ and a fixed pole $x_0$ in $\hn$ and $\mathbb{S}^{N-1}$ is the unit sphere in the $N$-dimensional euclidean space $\mathbb{R}^N$. Recall that the Riemannian Laplacian of a scalar function $u$ on $\hn$ is given by 
\begin{equation*}
	\Delta_{\hn} u (r, \Theta)  =
	\frac{1}{\sinh^2 r} \frac{\partial}{\partial r} \left[ (\sinh r)^{N-1} \frac{\partial u}{\partial r}(r, \Theta) \right] \\
	+ \frac{1}{\sinh^2 r} \Delta_{\mathbb{S}^{N-1}} u(r, \Theta),
\end{equation*}
where $\Delta_{\mathbb{S}^{N-1}}$ is the Riemannian Laplacian on the unit sphere $\mathbb{S}^{N-1}$. In particular,
the radial contribution of the Riemannian Laplacian, namely the operator involving only on $r$, $\Delta_{r,\hn} u$, reads as
\begin{equation*}
	\Delta_{r,\hn} u = \frac{1}{(\sinh r)^{N-1}} \frac{\partial}{\partial r} \left[ (\sinh r)^{N-1} \frac{\partial u}{\partial r}
	 \right] =  u^{\prime \prime} + (N-1) \coth r \,u^{\prime},
\end{equation*}
where from now on a prime will denote, for radial functions, derivative w.r.t .$r$. Also, let us recall the Gradient in terms of the polar coordinate decomposition is given by
\begin{equation*}
	\nabla_{\hn}u(r, \Theta)=\bigg(\frac{\partial u}{\partial r}(r, \Theta), \frac{1}{\sinh r}\nabla_{\mathbb{S}^{N-1}}u(r, \Theta)\bigg),
\end{equation*}
where $\nabla_{\mathbb{S}^{N-1}}$ denotes the Gradient on the unit sphere $\mathbb{S}^{N-1}$. Again, the radial contribution of the Gradient, $\nabla_{r,\hn}u$, is defined as 
\begin{equation*}
	\nabla_{r,\hn}u=\bigg(\frac{\partial u}{\partial r}, 0 \bigg).
\end{equation*}

\subsection{Hardy-Rellich and Poincar\'e identities and improved inequalities}
Our main result for radial operators reads as follows
\begin{theorem}\label{improved-hardy-radial}
For all $0\leq  \lambda \leq  \lambda_{1}(\hn)=\left(\frac{N-1}{2} \right)^2$ and all $u \in C_{c}^{\infty} (\hn \setminus \{ x_0 \})$ there holds
\begin{align*}
 &\int_{\hn}|\Delta_{r,\hn} u|^2\, \emph{d}v_{\hn} = \lambda \int_{\hn}  |\nabla_{r,\hn} u|^2\, \emph{d}v_{\hn} +h_N^2(\lambda) \int_{\hn} \frac{|\nabla_{r,\hn} u|^2}{r^2} \ \emph{d}v_{\hn} \\
 & +  \left[\left(\frac{N}{2} \right)^2 - h_N^2(\lambda)\right] \int_{\hn}  \frac{|\nabla_{r,\hn} u|^2}{\sinh^2 r} \ \emph{d}v_{\hn}+  \gamma_{N}(\lambda)\, h_N(\lambda)\int_{\hn} \frac{r \coth r - 1}{r^2}\, |\nabla_{r,\hn} u|^2  \ \emph{d}v_{\hn} \notag \\
 &+\int_{\hn}(\Psi_{\lambda}(r))^2\bigg|\nabla_{r,\hn}\bigg(\frac{u_r}{\Psi_{\lambda}(r)}\bigg)\bigg|^2\dv  \notag
\end{align*}
where $\gamma_{N}(\lambda):=\sqrt{(N-1)^2-4\lambda}$, $h_N(\lambda):=\frac{\gamma_{N}(\lambda)+1}{2}$ and $\Psi_{\lambda}(r) := r^{-\frac{N-2}{2}} \left(\frac{\sinh r}{r} \right)^{-\frac{N-1+\gamma_N(\lambda)}{2}}$. Furthermore, for $N\geq 5$ and $\lambda$ given, the constants $h_N^2(\lambda)$ and $ \left[\left(\frac{N}{2} \right)^2 - h_N^2(\lambda)\right]$ are jointly sharp in the sense that, fixed $h_N^2(\lambda)$, the inequality does not hold if we replace $ \left[\left(\frac{N}{2} \right)^2 - h_N^2(\lambda)\right]$ with a larger constant.
\end{theorem}
\begin{remark}
We remark that the the function $\frac{r \coth r - 1}{r^2}$ is positive, strictly decreasing and satisfies
$$\frac{r \coth r - 1}{r^2} \sim \frac{1}{3} \ \mbox{ as } \  r \rightarrow 0^+ \quad \mbox{ and }  \quad \frac{r \coth r - 1}{r^2} \sim \frac{1}{r} \  \mbox{ as }  \  r \rightarrow +\infty\,.$$
Furthermore, the map $[0,  \lambda_{1}(\hn)] \ni \lambda \mapsto h_N(\lambda)$ is decreasing and $\frac{1}{4}\leq h_N(\lambda)\leq \left(\frac{N}{2} \right)^2$.
\end{remark}

Furthemore, for non radial operators we obtain the second order analogous to \eqref{improved-poinc-lambda}:
\begin{theorem}\label{improved-hardy}
Let $N \geq 5.$ For all $0\leq  \lambda \leq  \lambda_{1}(\hn)=\left(\frac{N-1}{2} \right)^2$ and all $u \in C_{c}^{\infty} (\hn \setminus \{ x_0 \})$ there holds
\begin{align*}
 &\int_{\hn}|\Delta_{\hn} u|^2\, \emph{d}v_{\hn} \geq \lambda \int_{\hn}  |\nabla_{\hn} u|^2\, \emph{d}v_{\hn} +h_N^2(\lambda) \int_{\hn} \frac{|\nabla_{\hn} u|^2}{r^2} \ \emph{d}v_{\hn}  \\
 & +  \left[\left(\frac{N}{2} \right)^2 - h_N^2(\lambda)\right] \int_{\hn}  \frac{|\nabla_{\hn} u|^2}{\sinh^2 r} \ \emph{d}v_{\hn}+  \gamma_{N}(\lambda)\, h_N(\lambda)\int_{\hn} \frac{r \coth r - 1}{r^2}\, |\nabla_{\hn} u|^2  \ \emph{d}v_{\hn}\notag \\
 &+\int_{\hn}(\Psi_{\lambda}(r))^2\bigg|\nabla_{\hn}\bigg(\frac{u_r}{\Psi_{\lambda}(r)}\bigg)\bigg|^2\dv \notag
\end{align*}
where $\gamma_{N}(\lambda),h_N(\lambda)$ and $\Psi_{\lambda}(r)$ are as given in Theorem \ref{improved-hardy-radial}. Furthermore, for any given $\lambda$, the constants $h_N^2(\lambda)$ and $ \left[\left(\frac{N}{2} \right)^2 - h_N^2(\lambda)\right]$ are jointly sharp in the sense explained in Theorem \ref{improved-hardy-radial}.
\end{theorem}

 We notice that the dimension restriction $N\geq 5$ in Theorem \ref{improved-hardy} comes from assumption \eqref{extra_condition_2} in Theorem \ref{th_Bessel_rellich_non_rad} below where we state our abstract Rellich inequalities, see also Remark \ref{restriction} for some comments about this assumption that naturally comes when passing from the radial to the non radial framework. Theorems \ref{improved-hardy-radial} and \ref{improved-hardy} yield a number of improved Poincar\'e and Hardy-Rellich inequalities that we state here below; a comparison with previous results is provided in Remark \ref{comparison}. More precisely, for $\lambda=0$ we readily got the following improved Hardy-Rellich identity and inequality:

\begin{corollary}\label{1}
For all $u \in C_{c}^{\infty} (\hn \setminus \{ x_0 \})$ there holds
\begin{align*}
 \int_{\hn}|\Delta_{r,\hn} u|^2\, \emph{d}v_{\hn} &= \left(\frac{N}{2} \right)^2 \int_{\hn} \frac{|\nabla_{r,\hn} u|^2}{r^2} \ \emph{d}v_{\hn}\\
 & +  \frac{N(N-1)}{2}\int_{\hn} \frac{r \coth r - 1}{r^2}\, |\nabla_{r,\hn} u|^2  \ \emph{d}v_{\hn} \\
 &+\int_{\hn} \frac{r^{N}}{(\sinh r)^{2(N-1)} }\bigg|\nabla_{r,\hn}\bigg(\frac{(\sinh r)^{N-1} \, u_r}{r^{\frac{N}{2}}}\bigg)\bigg|^2\dv\,.
\end{align*}
 Moreover, if $N \geq 5$, for all $u \in C_{c}^{\infty} (\hn \setminus \{ x_0 \})$ there holds
\begin{align*}
 \int_{\hn}|\Delta_{\hn} u|^2\, \emph{d}v_{\hn} &\geq \left(\frac{N}{2} \right)^2 \int_{\hn} \frac{|\nabla_{\hn} u|^2}{r^2} \ \emph{d}v_{\hn}\\
 & +  \frac{N(N-1)}{2}\int_{\hn} \frac{r \coth r - 1}{r^2}\, |\nabla_{\hn} u|^2  \ \emph{d}v_{\hn} \\
&+\int_{\hn} \frac{r^{N}}{(\sinh r)^{2(N-1)} }\bigg|\nabla_{\hn}\bigg(\frac{(\sinh r)^{N-1} \, u_r}{r^{\frac{N}{2}}}\bigg)\bigg|^2\dv\,,
\end{align*}
 and the constant $ \left(\frac{N}{2} \right)^2$ appearing in the L.H.S of both equations is the sharp constant. 
\end{corollary}

For $\lambda=\lambda_{1}(\hn)$ we got an improvement of the second order Poincar\'e identity \eqref{higher_order_poin} with $l=0$, and the related inequality:

\begin{corollary}\label{3}
	For all $u\in C_c^\infty(\hn\setminus\{x_0\})$ there holds
	\begin{align*}
	\int_{\hn}|\Delta_{r,\hn} u|^2\dv&  =  \left(\frac{N-1}{2} \right)^2 \int_{\hn} |\nabla_{r,\hn}u|^2 \, \dv\\
		&+  \frac{1}{4} \int_{\hn} \frac{|\nabla_{r,\hn}u|^2}{r^2} \, \dv + \frac{N^2-1}{4} \int_{\hn} \frac{|\nabla_{\hn}u|^2}{\sinh^2 r} \, \dv \\
		&+\int_{\hn}\frac{r}{(\sinh r)^{N-1}}\bigg|\nabla_{r,\hn}\bigg(\frac{(\sinh r)^{\frac{N-1}{2}}u_r}{r^{\frac{1}{2}}}\bigg)\bigg|^2\dv.\notag
	\end{align*}
 Moreover, if $N \geq 5$, for all $u\in C_c^\infty(\hn\setminus\{x_0\})$ there holds
	\begin{align*}
		\int_{\hn}|\Delta_{\hn} u|^2\dv&  \geq  \left(\frac{N-1}{2} \right)^2 \int_{\hn} |\nabla_{\hn}u|^2 \, \dv\\
		&+  \frac{1}{4} \int_{\hn} \frac{|\nabla_{\hn}u|^2}{r^2} \, \dv + \frac{N^2-1}{4} \int_{\hn} \frac{|\nabla_{\hn}u|^2}{\sinh^2 r} \, \dv \\
		& + \, \int_{\hn}\frac{r}{(\sinh r)^{N-1}}\bigg|\nabla_{\hn}\bigg(\frac{(\sinh r)^{\frac{N-1}{2}}u_r}{r^{\frac{1}{2}}}\bigg)\bigg|^2\dv.\notag
	\end{align*}
	The constant $ \left(\frac{N-1}{2} \right)^2$ appearing in the L.H.S of both equations is the sharp constant. Moreover, for $N \geq 5,$ the constants $\frac{1}{4}$ and $\frac{N^2-1}{4}$ are jointly sharp in the sense explained in Theorem \ref{improved-hardy-radial}. 
\end{corollary} 

By combining Corollary \ref{1} with \cite[Corollary 3.2]{FLLM} we also get an improved Rellich inequality: 

\begin{corollary}\label{2} For all $u \in C_{c}^{\infty} (\hn \setminus \{ x_0 \})$ there holds
\begin{align*}
 \int_{\hn}|\Delta_{r,\hn} u|^2\, \emph{d}v_{\hn} &=  \frac{N^2}{4}\,\bigg(\frac{N-4}{2}\bigg)^{2} \int_{\hn} \frac{u^2}{r^4} \ \emph{d}v_{\hn}\\
 & +  \frac{N^2(N-4)(N-1)}{8}\int_{\hn} \frac{r \coth r - 1}{r^4}\, u^2  \ \emph{d}v_{\hn}\\
  & +\frac{N(N-1)}{2}\int_{\hn} \frac{r \coth r - 1}{r^2}\, |\nabla_{r,\hn} u|^2  \ \emph{d}v_{\hn} \\
 & +\frac{N^2}{4}\, \int_{\hn} \frac{1}{r^{N-2} }\bigg|\nabla_{r,\hn}\bigg(r^{\frac{N-4}{2}} u \bigg)\bigg|^2\dv \\
 &+\int_{\hn} \frac{r^{N}}{(\sinh r)^{2(N-1)} }\bigg|\nabla_{r,\hn}\bigg(\frac{(\sinh r)^{N-1} \, u_r}{r^{\frac{N}{2}}}\bigg)\bigg|^2\dv\,.
\end{align*}

Moreover, if $N \geq 5,$ for all $u \in C_{c}^{\infty} (\hn \setminus \{ x_0 \})$ there holds
\begin{align*}
 \int_{\hn}|\Delta_{\hn} u|^2\, \emph{d}v_{\hn} &\geq  \frac{N^2}{4}\,\bigg(\frac{N-4}{2}\bigg)^{2} \int_{\hn} \frac{u^2}{r^4} \ \emph{d}v_{\hn}\\
 & +  \frac{N^2(N-4)(N-1)}{8}\int_{\hn} \frac{r \coth r - 1}{r^4}\, u^2  \ \emph{d}v_{\hn}\\
 & +\frac{N(N-1)}{2}\int_{\hn} \frac{r \coth r - 1}{r^2}\, |\nabla_{\hn} u|^2  \ \emph{d}v_{\hn} \\
 & +\frac{N^2}{4}\, \int_{\hn} \frac{1}{r^{N-2} }\bigg|\nabla_{\hn}\bigg(r^{\frac{N-4}{2}} u \bigg)\bigg|^2\dv \\
 &+\int_{\hn} \frac{r^{N}}{(\sinh r)^{2(N-1)} }\bigg|\nabla_{\hn}\bigg(\frac{(\sinh r)^{N-1} \, u_r}{r^{\frac{N}{2}}}\bigg)\bigg|^2\dv\,,
\end{align*}
and the constant $\frac{N^2}{4}\,\left(\frac{N-4}{2}\right)^{2}$ appearing in the L.H.S of both equations is the sharp constant. 
\end{corollary}

Instead, by combining Corollary \ref{3} with \cite[Theorem 1.4 and Corollary 3.2]{FLLM}, we improve \eqref{higher_order_poin} with $l=0$, i.e. we complete the second order scenario about Poincar\'e identities and inequalities :

\begin{corollary}\label{4}
	For all $u\in C_c^\infty(\hn\setminus\{x_0\})$ there holds 
		
	\begin{align*}
	\int_{\hn}|\Delta_{r,\hn} u|^2\dv&  =  \left(\frac{N-1}{2} \right)^4 \int_{\hn} u^2 \dv \\
	&+\left(\frac{N-1}{4} \right)^2 \int_{\hn} \frac{u^2}{r^2} \dv+  \frac{(N-1)^3(N-3)}{16} \int_{\hn} \frac{u^2}{\sinh^2 r}  \dv\\
	&+  \frac{1}{4} \int_{\hn} \frac{|\nabla_{r,\hn}u|^2}{r^2} \, \dv+ \frac{N^2-1}{4} \int_{\hn} \frac{|\nabla_{r,\hn}u|^2}{\sinh^2 r} \, \dv \\
		&+\left[\left(\frac{N-1}{2} \right)^2+1\right]\int_{\hn}\frac{r}{(\sinh r)^{N-1}}\bigg|\nabla_{r,\hn}\bigg(\frac{(\sinh r)^{\frac{N-1}{2}}u_r}{r^{\frac{1}{2}}}\bigg)\bigg|^2\dv.\notag
	\end{align*}
	
Moreover, if $N \geq 5$, for all $u\in C_c^\infty(\hn\setminus\{x_0\})$ there holds
	\begin{align*}
		\int_{\hn}|\Delta_{\hn} u|^2\dv&  \geq  \left(\frac{N-1}{2} \right)^4 \int_{\hn} u^2 \dv \\
	&+\left(\frac{N-1}{4} \right)^2 \int_{\hn} \frac{u^2}{r^2} \dv+  \frac{(N-1)^3(N-3)}{16} \int_{\hn} \frac{u^2}{\sinh^2 r}  \dv\\
	&+  \frac{1}{4} \int_{\hn} \frac{|\nabla_{\hn}u|^2}{r^2} \, \dv+ \frac{N^2-1}{4} \int_{\hn} \frac{|\nabla_{\hn}u|^2}{\sinh^2 r} \, \dv \\
		&+\left[\left(\frac{N-1}{2} \right)^2+1\right]\int_{\hn}\frac{r}{(\sinh r)^{N-1}}\bigg|\nabla_{\hn}\bigg(\frac{(\sinh r)^{\frac{N-1}{2}}u_r}{r^{\frac{1}{2}}}\bigg)\bigg|^2\dv.\notag
	\end{align*}
	The constant $ \left(\frac{N-1}{2} \right)^4$ appearing in the L.H.S of both equations is the sharp constant. Moreover, for $N\geq 5$, the constants $\frac{1}{4}$ and $\frac{N^2-1}{4}$ in both equations are jointly sharp in the sense explained in Theorem \ref{improved-hardy-radial}. 
\end{corollary}

\begin{remark}\label{comparison}
As far as we are aware, improved second order Poincar\'e and Hardy-Rellich \emph{equalities} in $\hn$ were not known in literature; besides, the above inequalities yield improvements of Poincar\'e and Hardy-Rellich inequalities which are considerably stronger than those already known in literature. As concerns the Hardy-Rellich and Rellich inequalities, improved versions were already known from \cite{KR}, \cite{VHN} and \cite{YSK} on general manifolds but with fewer and smaller remainder terms. As a matter of example, if we compare Corollary \ref{1} with \cite[Theorem 4.2]{VHN}, the improvement of the Hardy-Rellich inequality provided there reads as $\frac{3N(N-1)}{2}\int_{\hn}\frac{|\nabla_{r,\hn} u|^2}{\pi^2+r^2}  \ \emph{d}v_{\hn}$, therefore it decays more rapidly, either as $r\rightarrow 0^+$ and as $r\rightarrow + \infty$, than the term $\frac{N(N-1)}{2}\int_{\hn} \frac{r \coth r - 1}{r^2}\, |\nabla_{r,\hn} u|^2  \ \emph{d}v_{\hn}$ provided in Corollary \ref{1}. Similarly, if we compare Corollary \ref{3} with \cite[Theorem 4.3]{VHN}, again, the corrections of the Rellich inequality provided there decays more rapidly than ours, either as $r\rightarrow 0^+$ and as $r\rightarrow  +\infty$. As concerns the improved second order Poincar\'e inequalities given by Corollaries \ref{2} and \ref{4}, the gain with respect to the inequalities already known in \cite{EGR} is in the adding of a further remainder term.
 
\end{remark}

\subsection{Second order Heisenberg-Pauli-Weyl uncertainty principle} \label{uncertainity}
Another remarkable consequence of Theorem \ref{improved-hardy} is the following quantitative version of HPW principle in $\hn$: 

\begin{theorem}\label{HPWhn}
 Let $N \geq 5.$ For all $0\leq \lambda \leq  \lambda_1(\hn)$ and all $u \in C_{c}^{\infty} (\hn \setminus \{ x_0 \})$ there holds
\begin{align}\label{eqHPWhn1}
 \left( \int_{\hn} \left( |\Delta_{\hn} u|^2 - \lambda |\nabla_{\hn} u|^2 \right) {\rm d}v_{\hn} \right)&\left( \int_{\hn} r^2 |\nabla_{\hn} u|^2  {\rm d}v_{\hn} \right)
   \\
\notag  \geq & \,  h_N^2(\lambda)  \left(  \int_{\hn}  |\nabla_{\hn} u|^2 {\rm d}v_{\hn} \right)^2 
\end{align}
 where $h_{N}(\lambda)$ is as defined as in Theorem  \ref{improved-hardy-radial}. In particular, for $\lambda=0,$ we obtain
\begin{equation}\label{eqHPWhn2}
 \left( \int_{\hn}  |\Delta_{\hn} u|^2  {\rm d}v_{\hn} \right)\, \left( \int_{\hn} r^2 |\nabla_{\hn} u|^2  \,  {\rm d}v_{\hn} \right)
\geq  \frac{N^2}{4}  \left(  \int_{\hn}  |\nabla_{\hn} u|^2 \, {\rm d}v_{\hn} \right)^2 \,,
\end{equation}
for all $u \in C_{c}^{\infty} (\hn \setminus \{ x_0 \})$.
 \end{theorem}

\begin{remark}	
 In the Euclidean context the second order Heisenberg-Pauli-Weyl uncertainty principle has been only recently studied in \cite[Theorem 2.1-2.2]{LamCazaFly} where it is proved that the best constant switches from $\frac{N^2}{4}$ to $\frac{(N+2)^2}{4}$ when passing to the second order. Moreover, Duong-Nguyen in \cite[Theorem 1.1]{newwgckn} has studied the weighted version of inequality \eqref{eqHPWhn2} in the Euclidean setting and discuss its sharp constants and extremals.

As far as we know, inequality \eqref{eqHPWhn1} is the first example of second order Heisenberg-Pauli-Weyl uncertainty principle in the Hyperbolic context. For the first order case, we refer instead to \cite{kristaly} and \cite{AKR} where the authors fully describe the influence of curvature to uncertainty principles in the
Riemannian and Finslerian settings. It's worth mentioning that a finer exploitation of Theorem~\ref{improved-hardy} yields the improved version of \eqref{eqHPWhn1} below which supports the conjecture that the sharp constant \eqref{eqHPWhn1} should be larger than $h_{N}^2(\lambda)$. More precisely, a small modification of the proof of Theorem \ref{HPWhn} allows us to prove that, for all $0\leq \lambda \leq  \lambda_1(\hn)$ and all $u \in C_{c}^{\infty} (\hn \setminus \{ x_0 \})$, there holds 
\begin{align*}
 &\left( \int_{\hn} \left( |\Delta_{\hn} u|^2 - \lambda |\nabla_{\hn} u|^2 \right) {\rm d}v_{\hn} \right)\left( \int_{\hn} r^2 |\nabla_{\hn} u|^2  {\rm d}v_{\hn} \right)
   \\
\notag & \quad \geq h_N^2(\lambda)  \left(  \int_{\hn}  |\nabla_{\hn} u|^2 {\rm d}v_{\hn} \right)^2 + \left( \int_{\hn} r^2 |\nabla_{\hn} u|^2  {\rm d}v_{\hn} \right) \times\\ \notag
&\times \left\{ \left[\left(\frac{N}{2} \right)^2 - h_N^2(\lambda)\right] \int_{\hn}  \frac{|\nabla_{\hn} u|^2}{\sinh^2 r} \ \emph{d}v_{\hn}+  \gamma_{N}(\lambda)\, h_N(\lambda)\int_{\hn} \frac{r \coth r - 1}{r^2}\, |\nabla_{\hn} u|^2  \ \emph{d}v_{\hn} \right\} \notag
\end{align*}
where $\gamma_{N}(\lambda)$ and $h_{N}(\lambda)$ are defined as in Theorem  \ref{improved-hardy-radial}. Therefore, for $\lambda=0,$ we obtain the improved version of \eqref{eqHPWhn2}:
\begin{align*}
 &
 \left( \int_{\hn}  |\Delta_{\hn} u|^2  {\rm d}v_{\hn} \right)\, \left( \int_{\hn} r^2 |\nabla_{\hn} u|^2  \,  {\rm d}v_{\hn} \right)
\geq  \frac{N^2}{4}  \left(  \int_{\hn}  |\nabla_{\hn} u|^2 \, {\rm d}v_{\hn} \right)^2 \\
&+ \left( \int_{\hn} r^2 |\nabla_{\hn} u|^2  {\rm d}v_{\hn} \right) \left( \frac{N(N-1)}{2}\int_{\hn} \frac{r \coth r - 1}{r^2}\, |\nabla_{\hn} u|^2  \ \emph{d}v_{\hn} \right)
\end{align*}
for all $u \in C_{c}^{\infty} (\hn \setminus \{ x_0 \})$. The above inequality should be compared with inequality \eqref{eqHPWhn3} provided in Section \ref{bessel} which also improves \eqref{eqHPWhn2}.

\end{remark}

We conclude the section by stating the counterpart of Theorem \ref{HPWhn} for radial operators:

\begin{theorem}
For all $0\leq \lambda \leq  \lambda_1(\hn)$ and all $u \in C_{c}^{\infty} (\hn \setminus \{ x_0 \})$ there holds
\begin{align*}
 \left( \int_{\hn} \left( |\Delta_{r,\hn} u|^2 - \lambda |\nabla_{r,\hn} u|^2 \right) {\rm d}v_{\hn} \right)&\left( \int_{\hn} r^2 |\nabla_{r,\hn} u|^2  {\rm d}v_{\hn} \right)
   \\
\notag  \geq & \,  h_N^2(\lambda)  \left(  \int_{\hn}  |\nabla_{r,\hn} u|^2 {\rm d}v_{\hn} \right)^2 
\end{align*}
 where $h_{N}(\lambda)$ is as defined as in Theorem  \ref{improved-hardy-radial}. In particular, for $\lambda=0,$ we obtain
\begin{equation*}
 \left( \int_{\hn}  |\Delta_{r,\hn} u|^2  {\rm d}v_{\hn} \right)\, \left( \int_{\hn} r^2 |\nabla_{r,\hn} u|^2  \,  {\rm d}v_{\hn} \right)
\geq  \frac{N^2}{4}  \left(  \int_{\hn}  |\nabla_{r,\hn} u|^2 \, {\rm d}v_{\hn} \right)^2 
\end{equation*}
for all $u \in C_{c}^{\infty} (\hn \setminus \{ x_0 \})$.
 \end{theorem}

\section{Abstract Rellich identities and inequalities via Bessel pairs}\label{bessel}

Ghoussoub-Moradifam in \cite{GM} provided a very general framework to obtain various Hardy-type 
inequalities and their improvements on the Euclidean space (or  bounded domain). It was based on the notion of Bessel pairs that we recall in the following
\begin{definition}
	We say that a pair $(V,W)$ of $C^1$-functions is a Bessel pair on $(0,R)$ for some $0<R\leq \infty$  if the ordinary differential equation:
	\begin{equation*}
		(Vy')'+Wy=0
	\end{equation*} 
admits a positive solutions $f$ on the interval $(0,R)$.
\end{definition}
In \cite{GM} the authors proved the following inequality for some positive constant $C >0:$
\begin{equation}\label{1eq}
\int_{B} V(x) |\nabla u|^2 \, {\rm d}x \geq C \int_{B} W(x) \, |u|^2 \, {\rm d}x \quad \forall \ u \in C_c^{\infty}(B),
\end{equation}
subject to the constraints that the functions $V$ and $W$ are positive radial functions defined on the euclidean ball $B$ and such that: $(r^{N-1}V, r^{N-1}W)$ is a Bessel pairs  $\int_{0}^{R} \frac{1}{r^{N-1} V(r)} \, {\rm d}r = \infty$ and $\int_{0}^{R} r^{N-1} V(r) \, {\rm d}r < \infty$ where $0<R\leq \infty$ is the radius of the ball $B$.

In view of \eqref{1eq}, with particular choices of $(V, W)$, the results in \cite{GM} simplified and improved several known results concerning Hardy inequalities 
and theirs improvements. Recently, the notion of Bessel pairs has been exploited: in \cite{LLZ} to establish improved Hardy inequalities
involving general distance functions, in \cite{LLZ1} to sharpen several
 Hardy type inequalities on upper half spaces, and in \cite{NL} to prove Hardy inequalities on Homogeneous 
 groups. 

\medskip 

Regarding Cartan-Hadamard manifolds, the notion of Bessel pairs has been very recently exploited to obtain improved Hardy inequalities in \cite{FLLM}; to our future purposes, we recall their Theorem 3.2 on $\hn$:
\begin{lemma}\label{th_Bessel_hardy} \cite[Theorem 3.2]{FLLM}
	Let $(r^{N-1}V,r^{N-1}W)$ be a Bessel pair on $(0,R)$ with positive solution $f$ on $(0,R)$. Then for all $u\in C_c^{\infty}(\hn\setminus\{x_0\})$, there holds
		\begin{align*}
		\int_{B_R}V(r)|\nabla_{\hn} u|^2\dv&= \int_{B_R}W(r)|u|^2\dv+\int_{B_R}V(r)(f(r))^2\bigg|\nabla_{\hn}\bigg(\frac{u}{f(r)}\bigg)\bigg|^2\dv \notag \\&-(N-1)\int_{B_R}V(r)\frac{f'(r)}{f(r)}\bigg(\coth r - \frac{1}{r}\bigg)u^2\dv.
	\end{align*}
and
\begin{align*}
		\int_{B_R}V(r)|\nabla_{r,\hn} u|^2\dv&= \int_{B_R}W(r)|u|^2\dv+\int_{B_R}V(r)(f(r))^2\bigg|\nabla_{r,\hn}\bigg(\frac{u}{f(r)}\bigg)\bigg|^2\dv \notag \\&-(N-1)\int_{B_R}V(r)\frac{f'(r)}{f(r)}\bigg(\coth r - \frac{1}{r}\bigg)u^2\dv.
	\end{align*}

\end{lemma}

\medskip 

In view of Lemma \ref{th_Bessel_hardy} a natural question is to study higher order Hardy type inequalities in $\hn$ using the notion of Bessel pairs. In the Euclidean space (or in bounded euclidean domains) these questions were studied in \cite{GM}. One of their results read as follows: let $0<R\leq\infty, V \text{ and } W$ be positive $C^1$-functions on $B_R\setminus\{0\}$ such that $(r^{N-1}V,r^{N-1}W)$ forms a Bessel pair.
  Then for all radial functions $u\in C_0^\infty(B_R)$ there holds
\begin{align}\label{develop}
	\int_{B_R} V(x) |\Delta u|^2 \, {\rm d}x \geq  \int_{B} W(x) \, |\nabla u|^2 \, {\rm d}x+(N-1)\int_{B_R}\bigg(\frac{V(x)}{|x|^2}-\frac{V_r(x)}{|x|}\bigg)|\nabla u|^2 \, {\rm d}x,
\end{align}
where $r=|x|$. In addition, if $W(x)-2\frac{V(x)}{|x|^2}+2\frac{V_r(x)}{|x|}-V_{rr}(x)\geq 0$ on $(0,R)$, then the above is true for non radial function as well  (we refer \cite[Theorem 3.1-3.3]{GM} for more insight).
We also refer to \cite{DLT, DLT1, NL1} for recent results on Hardy-Rellich inequalities and their improvements on the Euclidean space using the approach of Bessel pairs.

In the present article, we extend \eqref{develop} to $\hn$ by showing first the following:

\begin{theorem}\label{th_Bessel_rellich}
	Let $(r^{N-1}V,r^{N-1}W)$ be a Bessel pair on $(0,R)$ with positive solution $f$ on $(0,R)$. Then for all radial function $u\in C_c^\infty(B_R\setminus\{x_0\})$ there holds
	\begin{align}\label{eq_Bessel_rellich}
		\int_{B_R}V(r)|\Delta_{\hn} u|^2\dv&= \int_{B_R}W(r)|\nabla_{\hn} u|^2\dv  \notag \\ 
		&+(N-1)\int_{B_R}\bigg(\frac{V(r)}{\sinh^2 r}-\frac{V_{r}(r)\cosh r}{\sinh r}\bigg)|\nabla_{\hn} u|^2\dv \notag \\
		&-(N-1)\int_{B_R}V(r)\frac{f^\prime}{f}\bigg(\coth r - \frac{1}{r}\bigg)|\nabla_{\hn} u|^2\dv \notag \\
		&+\int_{B_R}V(r)(f(r))^2\bigg|\nabla_{\hn}\bigg(\frac{u_r}{f(r)}\bigg)\bigg|^2\dv\,.
	\end{align}
\end{theorem}

As a direct consequence of the above result, we tackle the non-radial scenario by spherical harmonic method and we prove:
\begin{corollary}\label{rop_th_Bessel_rellich}
	Let $(r^{N-1}V,r^{N-1}W)$ be a Bessel pair on $(0,\infty)$ with positive solution $f$ on $(0,\infty)$. Then for all $u\in C_c^\infty(\hn\setminus\{x_0\})$ there holds
	\begin{align*}
		\int_{\hn}V(r)|\Delta_{r,\hn} u|^2\dv&= \int_{\hn}W(r)|\nabla_{r,\hn} u|^2\dv \notag \\ 
		&+(N-1)\int_{\hn}\bigg(\frac{V(r)}{\sinh^2 r}-\frac{V_{r}(r)\cosh r}{\sinh r}\bigg)|\nabla_{r,\hn} u|^2\dv \notag \\
		&-(N-1)\int_{\hn}V(r)\frac{f^\prime}{f}\bigg(\coth r - \frac{1}{r}\bigg)|\nabla_{r,\hn} u|^2\dv\notag\\
		&+\int_{\hn}V(r)(f(r))^2\bigg|\nabla_{r,\hn}\bigg(\frac{u_r}{f(r)}\bigg)\bigg|^2\dv
			\end{align*}
\end{corollary}

\medskip

Now it is natural to ask whether there is a counterpart of Theorem \ref{th_Bessel_rellich} for any function, not necessarily radial. We give an affirmative answer in below provided that $V$ satisfies the extra condition \eqref{extra_condition_2} below:
\begin{theorem}\label{th_Bessel_rellich_non_rad}
	Let $(r^{N-1}V,r^{N-1}W)$ be a Bessel pair on $(0,\infty)$ with positive solution $f$ on $(0,\infty)$. Also assume $N\geq 5$ and $V$ satisfies \begin{align}\label{extra_condition_2}
		(N-5)\frac{V(r)}{\sinh^2 r}+3\frac{V_r(r)\cosh r}{\sinh r}-V_{rr}(r)+(N-4)V(r)\geq 0.
	\end{align}
	Then for all $u\in C_c^\infty(\hn\setminus\{x_0\})$ there holds
	\begin{align}\label{eq_nrad_rellich}
		\int_{\hn}V(r)|\Delta_{\hn} u|^2\dv&\geq \int_{\hn}W(r)|\nabla_{\hn} u|^2\dv \notag \\ 
		&+(N-1)\int_{\hn}\bigg(\frac{V(r)}{\sinh^2 r}-\frac{V_{r}(r)\cosh r}{\sinh r}\bigg)|\nabla_{\hn} u|^2\dv \notag \\ 
		&-(N-1)\int_{\hn}V(r)\frac{f^\prime}{f}\bigg(\coth r - \frac{1}{r}\bigg)|\nabla_{\hn} u|^2\dv\notag\\
		&+\int_{\hn}V(r)(f(r))^2\bigg|\nabla_{\hn}\bigg(\frac{u_r}{f(r)}\bigg)\bigg|^2\dv\,.
			\end{align}
\end{theorem}

\begin{remark}\label{restriction}
We remark that assumption \eqref{extra_condition_2} in Theorem \ref{th_Bessel_rellich_non_rad} is not too restrictive to our purposes: we shall provide a remarkable family of $(V,W)$ for which the assumption holds true in the proof of Theorem \ref{improved-hardy-radial}. On the other hand, an analogous assumption was even required in the Euclidean space as well, see \eqref{develop} and the comments just below; this seems the natural prize to pay in order to pass to the higher order case.
\end{remark}

We conclude the section by stating an abstract version of  Heisenberg-Pauli-Weyl uncertainty principle involving Bessel pairs which follows as a corollary from Theorem \ref{th_Bessel_rellich_non_rad} and Corollary \ref{rop_th_Bessel_rellich}:
\begin{theorem}\label{rellich_usp}
	Let $(r^{N-1}V,r^{N-1}W)$ be a Bessel pair on $(0,\infty)$ with positive solution $f$ on $(0,\infty)$ and set
	\begin{equation*}
		\tilde W(r):=W(r)+(N-1)\bigg(\frac{V(r)}{\sinh^2 r}-\frac{V_{r}(r)\cosh r}{\sinh r}\bigg)-(N-1)V(r)\frac{f^\prime}{f}\bigg(\coth r - \frac{1}{r}\bigg).
	\end{equation*}
	Furthermore, let $N\geq 5$ and assume that $V$ satisfies \ref{extra_condition_2} and that $\tilde W(r)> 0$ for all $r>0$. Then for all $u\in C_c^{\infty}(\hn\setminus\{x_0\})$ there holds
	\begin{equation*}
		\bigg(\int_{\hn}V(r)|\Delta_{\hn} u|^2\dv\bigg)\bigg(\int_{\hn}\frac{|\nabla_{\hn}u|^2}{\tilde W(r)}\dv\bigg)\geq \bigg(\int_{\hn}|\nabla_{\hn}u|^2\dv\bigg)^2
	\end{equation*}
	and \begin{equation*}
		\bigg(\int_{\hn}V(r)|\Delta_{r,\hn} u|^2\dv\bigg)\bigg(\int_{\hn}\frac{|\nabla_{r,\hn}u|^2}{\tilde W(r)}\dv\bigg)\geq \bigg(\int_{\hn} |\nabla_{r,\hn}u|^2\dv\bigg)^2.
	\end{equation*}
\end{theorem}
We want to mention that for the second inequality we do not require condition \ref{extra_condition_2}, whereas the other conditions and Corollary \ref{rop_th_Bessel_rellich} are enough.

\begin{remark}
A non trivial example of pairs satisfying the assumptions of Theorem \ref{rellich_usp} is given by the family of Bessel pairs $(r^{N-1},r^{N-1}W_{\lambda})$, for all $0\leq  \lambda \leq  \lambda_{1}(\hn)$, defined in \eqref{potential} below and exploited in the proof of Theorem \ref{improved-hardy-radial}. Indeed, they satisfy condition \eqref{extra_condition_2} and, in this case, the function $\tilde W$ reads
$$ \tilde W_{\lambda}(r)=\lambda+h_N^2(\lambda) \frac{1}{r^2}+ \left(\left(\frac{N}{2} \right)^2-h_N^2(\lambda)\right) \frac{1}{\sinh^2 r}+ \frac{\gamma_{N}(\lambda)\,h_N(\lambda)}{r} \bigg(\coth r - \frac{1}{r}\bigg)
$$
which is positive in $(0,+\infty)$ for all $0\leq  \lambda \leq  \lambda_{1}(\hn)$. In particular, with this pair, taking $\lambda=0$ for simplicity, Theorem \ref{rellich_usp} yields
\begin{equation}\label{eqHPWhn3}
		\bigg(\int_{\hn}|\Delta_{\hn} u|^2\dv\bigg)\bigg(\int_{\hn}\frac{|\nabla_{\hn}u|^2}{  \frac{N^2}{4} \frac{1}{r^2}+   \frac{N(N-1)}{2r}(\coth r - \frac{1}{r})}\dv\bigg)\geq \bigg(\int_{\hn}|\nabla_{\hn}u|^2\dv\bigg)^2\,,
	\end{equation}
	for all $u\in C_c^{\infty}(\hn\setminus\{x_0\})$. The above inequality turns out to be more stringent than \eqref{eqHPWhn2} thereby confirming the conjecture that $\frac{N^2}{4}$ is not the sharp constant in \eqref{eqHPWhn2}.
\end{remark}


\section{Proofs of Theorems \ref{improved-hardy-radial}, \ref{improved-hardy}, \ref{HPWhn} and Corollaries \ref{2},\ref{4}}\label{applications}
{\bf  Proofs of Theorems \ref{improved-hardy-radial} and \ref{improved-hardy}}. The proof follows, respectively, by applying Corollary \ref{rop_th_Bessel_rellich} and Theorem \ref{th_Bessel_rellich_non_rad} with the family of Bessel pairs $(r^{N-1},r^{N-1}W_{\lambda})$ with $0\leq  \lambda \leq  \lambda_{1}(\hn)$ and
\begin{align}\label{potential}
W_{\lambda}(r)&:=\lambda+h_N^2(\lambda) \frac{1}{r^2}+ \left(\left(\frac{N-2}{2} \right)^2-h_N^2(\lambda)\right)   \frac{1}{\sinh^2 r}\notag \\&+ \left( \frac{\gamma_{N}(\lambda)\,h_N(\lambda)}{r}+(N-1) \frac{\Psi_{\lambda}'(r)}{\Psi_{\lambda}(r)}\right) \bigg(\coth r - \frac{1}{r}\bigg) \qquad (r>0)\,,
\end{align}
where $ \gamma_{N}(\lambda)$ and $h_N(\lambda)$ are as defined in the statement of Theorem \ref{th_Bessel_rellich} and
$$
\Psi_{\lambda}(r) := r^{-\frac{N-2}{2}} \left(\frac{\sinh r}{r} \right)^{-\frac{N-1+\gamma_N(\lambda)}{2}}\qquad (r>0)\,.
$$
In particular, by noticing that
$$
\Psi'_{\lambda}(r) = \Psi_{\lambda}(r)\left[ \frac{h_N(\lambda)}{r}+\frac{1-N-\gamma_N(\lambda)}{2} \coth r \right]\,,
$$
\begin{align*}
\Psi''_{\lambda}(r) &= \Psi_{\lambda}(r)\Big[ \frac{(1-N-\gamma_N(\lambda))^2}{4}+\frac{\gamma_N^2(\lambda)-1}{r^2}\\
&-\frac{(1-N-\gamma_N(\lambda))(1+N+\gamma_N(\lambda))}{4 \sinh^2 r} + \frac{(1-N-\gamma_N(\lambda)) h_N(\lambda) \coth r}{r} \Big]
\end{align*}
and recalling the definition of $\gamma_N(\lambda)$, it follows that $\Psi_{\lambda}(r)$ satisfies
$$(r^{N-1} \Psi_{\lambda}'(r))'+r^{N-1}W_{\lambda}(r) \Psi_{\lambda}(r)=0 \quad \text{for } r>0\,,$$
 namely $(r^{N-1},r^{N-1}W_{\lambda})$ is a Bessel pair with positive solution $\Psi_{\lambda}(r)$. See also \cite[Lemma 6.2]{BGGP} where the functions $\Psi_{\lambda}$ were originally introduced but exploited with different purposes.
Finally, from Corollary \ref{rop_th_Bessel_rellich} we deduce that, for all function $u\in C_c^\infty(B_R\setminus\{x_0\})$, there holds
	\begin{align*}
		\int_{B_R}|\Delta_{\hn} u|^2\dv&= \int_{B_R}W_{\lambda}(r)|\nabla_{\hn} u|^2\dv  \notag \\ 
		&+(N-1)\int_{B_R}\bigg(\frac{1}{\sinh^2 r}\bigg)|\nabla_{\hn} u|^2\dv \notag \\
		&-(N-1)\int_{B_R} \frac{\Psi_{\lambda}'(r)}{\Psi_{\lambda}(r)} \bigg(\coth r - \frac{1}{r}\bigg)|\nabla_{\hn} u|^2\dv \notag \\
		&+\int_{B_R}(\Psi_{\lambda}(r))^2\bigg|\nabla_{\hn}\bigg(\frac{u_r}{\Psi_{\lambda}(r)}\bigg)\bigg|^2\dv\,.
	\end{align*}
By this, recalling \eqref{potential}, the proof of Theorem \ref{improved-hardy-radial} follows. The proof of Theorem \ref{improved-hardy} works similarly by applying Theorem \ref{th_Bessel_rellich_non_rad} since condition \eqref{extra_condition_2} holds for the Bessel pair $(r^{N-1},r^{N-1}W_{\lambda})$ if $N\geq 5$.

As concerns the proof of the fact that the constants $h_N^2(\lambda)$ and $ \left[\left(\frac{N}{2} \right)^2 - h_N^2(\lambda)\right]$ are jointly sharp when $N\geq 5$, this follows by noticing that as $r\rightarrow 0$ we have 
 \begin{equation*}
	h_N^2(\lambda) \int_{\hn}\frac{|\nabla_{\hn} u|^2}{r^2}\dv+ \left[\left(\frac{N}{2} \right)^2 - h_N^2(\lambda)\right]\int_{\hn}\frac{|\nabla_{\hn} u|^2}{\sinh^2 r}\dv \sim \frac{N^2}{4}  \int_{\hn}\frac{|\nabla_{\hn} u|^2}{r^2}\dv\,.
	\end{equation*}
 Therefore, locally, we recover inequality \eqref{rellich-general} for $l=1$;  by this we readily infer that, for $h_N^2(\lambda)$ fixed, any larger constant in front of the term $\frac{|\nabla_{\hn} u|^2}{\sinh^2 r }$ would contradict the optimality of the constant $\frac{N^2}{4}$ in \eqref{rellich-general} (when $l=1$). \hfill $\Box$

\par
\bigskip
\par

{\bf Proof of  Corollary \ref{2}.}
The proof follows from Corollary \ref{1} by evaluating the term $\int_{\hn}\frac{|\nabla_{\hn} u|^2}{r^2}  \dv$ with the aid of \cite[Corollary 3.2]{FLLM} from which we know that
\begin{align*}
\int_{\hn}\frac{|\nabla_{\hn}u|^2}{r^2} \dv &= \bigg(\frac{N-4}{2}\bigg)^{2} \int_{\hn} \frac{u^2}{r^4} \dv \\
 & +  \frac{(N-4)(N-1)}{2}\int_{\hn} \frac{r \coth r - 1}{r^4}\, u^2   \dv\\
 & + \int_{\hn} \frac{1}{r^{N-2} }\bigg|\nabla_{\hn}\bigg(r^{\frac{N-4}{2}} u \bigg)\bigg|^2\dv\,.
\end{align*}
for all $u \in C_{c}^{\infty} (\hn \setminus \{ x_0 \})$. The proof for radial operators follows similarly since the above identity holds with the same constants for radial operators too. \hfill $\Box$

\par
\bigskip
\par

{\bf Proof of  Corollary \ref{4}.}
Here the proof follows by combining Corollary \ref{3} with \cite[Theorem 1.4]{FLLM} according to which we know that
\begin{align*}
&\int_{\hn} |\nabla_{\hn} u|^2 \dv=  \left(\frac{N-1}{2} \right)^2 \int_{\hn} u^2 \dv \notag \\
&+ \frac{1}{4} \int_{\hn} \frac{u^2}{r^2} \dv + \frac{(N-1)(N-3)}{4} \int_{\hn} \frac{u^2}{\sinh^2 r}  \dv\\
&+\int_{\hn}\frac{r}{(\sinh r)^{N-1}}\bigg|\nabla_{\hn}\bigg(\frac{(\sinh r)^{\frac{N-1}{2}}u_r}{r^{\frac{1}{2}}}\bigg)\bigg|^2\dv\,.
\end{align*}
for all $u \in C_{c}^{\infty} (\hn \setminus \{ x_0 \})$ and similarly for radial operators since the above identity holds with the same constants for radial operators too. \hfill $\Box$

\par
\bigskip
\par

{\bf Proof of Theorem \ref{HPWhn}.}
The proof is a simple application of Cauchy-Schwartz inequality combined with Theorem~\ref{improved-hardy}:
\begin{align*}
\int_{\hn} |\nabla_{\hn} u|^2 \, {\rm d}v_{\hn} &= \int_{\hn} r |\nabla_{\hn}u| \frac{|\nabla_{\hn}u|}{r} \, {\rm d}v_{\hn} \\
& \leq \left( \int_{\hn} r^2 |\nabla_{\hn} u|^2 \, {\rm d}v_{\hn} \right)^{\frac{1}{2}} \underbrace{\left( \int_{\hn} \frac{|\nabla_{\hn} u|^2}{r^2} \, {\rm d}v_{\hn}\right)^{\frac{1}{2}}}_{Using \ Theorem~\ref{improved-hardy}} \\
& \leq \frac{1}{h_{N} (\lambda)} \left( \int_{\hn} \left( |\Delta_{\hn} u|^2 - \lambda |\nabla_{\hn} u|^2 \right) {\rm d}v_{\hn} \right)^{\frac{1}{2}}\left( \int_{\hn} r^2 |\nabla_{\hn} u|^2  {\rm d}v_{\hn} \right)^{\frac{1}{2}}.
\end{align*}
 \hfill $\Box$

\section{Proofs of Theorem~\ref{th_Bessel_rellich}, Corollary \ref{rop_th_Bessel_rellich}, Theorem~\ref{th_Bessel_rellich_non_rad} and Theorem \ref{rellich_usp}}\label{bessel proofs}

We shall begin with the proof of Theorem~\ref{th_Bessel_rellich}. 

\medskip 

{\bf Proof of Theorem~\ref{th_Bessel_rellich}.}

	Let $u\in C_{c}^{\infty}(B_R\setminus\{x_0\})$ be a radial function, in terms of polar coordinates we have
	\begin{align*}
		&\int_{B_R}V(r)|\Delta_{\hn} u|^2\dv=N\omega_N\bigg[\int_{0}^{R}V(r)u_{rr}^2(\sinh r)^{N-1}\dr\\
		& +(N-1)^2\int_{0}^{R}V(r)(\coth r)^2u_r^2(\sinh r)^{N-1}\dr\\
		& +2(N-1)\int_{0}^{R}V(r)u_{rr}u_{r}(\coth r)(\sinh r)^{N-1}\dr \bigg].
	\end{align*}
	
	Now, applying integration by parts in the last term and setting $\nu=u_r$, we deduce
	{\begin{align}\label{nu}
		&\int_{B_R}V(r)|\Delta_{\hn} u|^2\dv=\int_{B_R}V(r)|\nabla_{\hn} \nu|^2\dv \notag\\
		&+(N-1)\int_{B_R}\bigg(\frac{V(r)}{\sinh^2 r}-\frac{V_{r}(r)\cosh r}{\sinh r}\bigg)|\nu|^2\dv.
	\end{align}
	
	On the other hand, from Lemma~\ref{th_Bessel_hardy} for the function $\nu$ we have
	\begin{align*}
		\int_{B_R}V(r)|\nabla_{\hn} \nu|^2\dv&= \int_{B_R}W(r)|\nu|^2\dv+\int_{B_R}V(r)(f(r))^2\bigg|\nabla_{\hn}\bigg(\frac{\nu}{f(r)}\bigg)\bigg|^2\dv\\&-(N-1)\int_{B_R}V(r)\frac{f^\prime}{f}\bigg(\coth r - \frac{1}{r}\bigg)|\nu|^2\dv.
	\end{align*}
	
	By using this identity into \eqref{nu} and writing back in terms of $u$ we deduce \eqref{eq_Bessel_rellich}.
\hfill $\Box$

\medskip 

{\bf Spherical harmonics.}

Before going to prove Corollary \ref{rop_th_Bessel_rellich} and Theorem~\ref{th_Bessel_rellich_non_rad}, we shall mention some useful facts from spherical harmonics. 

\medskip 

Let $u(x)=u(r,\Theta)\in C_{c}^\infty(\hn)$, $r\in({0},\infty)$ and $\Theta\in \mathbb{S}^{N-1}$, we can write
\begin{equation}\label{decomp}
	u(r,\Theta)=\sum_{n=0}^{\infty}a_n(r)P_n(\Theta)
\end{equation}
in $L^2(\hn)$, where $\{ P_n \}$ is an orthonormal system of spherical harmonics and 
\begin{equation*}
	a_n(r)=\int_{\mathbb{S}^{N-1}}u(r,\Theta)P_n(\Theta) \ {\rm d}\Theta\,.
\end{equation*} 
A spherical harmonic $P_n$ of order $n$ is the restriction to $\mathbb{S}^{N-1}$ of a homogeneous harmonic polynomial of degree $n.$ Moreover, it
satisfies $$-\Delta_{\mathbb{S}^{N-1}}P_n=\lambda_n P_n$$
for all $n\in\mathbb{N}\cup\{0\}$, where $\lambda_n=(n^2+(N-2)n)$ are the eigenvalues of Laplace Beltrami operator $-\Delta_{\mathbb{S}^{N-1}}$ on $\mathbb{S}^{N-1}$ with corresponding eigenspace dimension $c_n$. We note that $\lambda_n\geq N-1$ for $n\geq 1$, $\lambda_0=0$, $c_0=1$, $c_1=N$ and for $n\geq 2$ 
\begin{equation*}
	c_n=\binom{N+n-1}{n}-\binom{N+n-3}{n-2}.
\end{equation*}

\medspace

In a continuation let us also describe the Gradient and Laplace Beltrami operator in this setting. Now onward, to shorten the notations, we will always assume $\psi(r)=\sinh r.$  They will look like as follows:
\begin{align*}
|\nabla_{\hn}u|^2 =  \sum_{n = 0}^{\infty} {a_{n}^{\prime}}^2 P_{n}^2  + \frac{a_{n}^2}{\psi^2}  |\nabla_{\mathbb{S}^{N-1}}P_{n}|^2
\end{align*}
and
\begin{align}\label{sphr_lap}
	(\Delta_{\hn} u)^2 & = \sum_{n = 0}^{\infty} \left( a_{n}^{\prime \prime} + (N-1) \frac{\psi^{\prime}}{\psi} a_{n}^{\prime} \right)^2 P_{n}^2 + 
	\sum_{n = 0}^{\infty} \frac{a_{n}^2}{\psi^4} (\Delta_{\mathbb{S}^{N-1}} P_{n})^2 \\
	& + 2 \sum_{n = 0}^{\infty} \left( a_{n}^{\prime \prime} + (N-1) \frac{\psi^{\prime}}{\psi} a_{n}^{\prime} \right) \frac{a_{n}}{\psi^2} (\Delta_{\mathbb{S}^{N-1}} P_{n}) P_{n}.\notag
\end{align}

Along with this the radial contribution of the operators will look like as follows:
\begin{align*}
	|\nabla_{r,\hn}u|^2 =  \sum_{n = 0}^{\infty} {a_{n}^{\prime}}^2 P_{n}^2	
\end{align*}
and
\begin{align*}
	(\Delta_{r,\hn} u)^2 & = \sum_{n = 0}^{\infty} \left( a_{n}^{\prime \prime} + (N-1) \frac{\psi^{\prime}}{\psi} a_{n}^{\prime} \right)^2 P_{n}^2\,.
\end{align*}

\par \bigskip \par

{\bf Proof of Corollary \ref{rop_th_Bessel_rellich}.}
By spherical harmonics, we decompose $u$ as in  \eqref{decomp}. Now, exploiting Theorem \ref{th_Bessel_rellich} for each $a_n$, we deduce
	\begin{align*}
		&\int_{\hn}V(r)|\Delta_{r,\hn} u|^2\dv =\sum_{n = 0}^{\infty} \int_{0}^{\infty} V(r) \left( a_{n}^{\prime \prime} + (N-1) \frac{\psi^{\prime}}{\psi} a_{n}^{\prime} \right)^2 \psi^{N-1}\dr \\
		&=\sum_{n = 0}^{\infty} \bigg[\int_{0}^{\infty} W{a_n^{\prime}}^2\psi^{N-1}\dr+\int_{0}^{\infty}Vf^2\left[\left(\frac{a_n'}{f}\right)'\right]^2  \psi^{N-1}\dr\\&-(N-1)\int_{0}^{\infty}V\frac{f^\prime}{f}\bigg(\coth r - \frac{1}{r}\bigg)(a_n^{\prime})^2\psi^{N-1}\dr\\&+(N-1)\int_{0}^{\infty}V{a_n^{\prime}}^2\psi^{N-3}\dr-(N-1)\int_{0}^{\infty}V_r\psi^{\prime}(a_n')^2\psi^{N-2}\dr\bigg]\\
		&= \int_{\hn}W(r)|\nabla_{r,\hn} u|^2\dv +\int_{\hn}V(r)(f(r))^2\bigg|\nabla_{r,\hn}\bigg(\frac{u_r}{f(r)}\bigg)\bigg|^2\dv \\ 
		&-(N-1)\int_{\hn}V(r)\frac{f^\prime}{f}\bigg(\coth r - \frac{1}{r}\bigg)|\nabla_{r,\hn} u|^2\dv \\
		&+(N-1)\int_{\hn}\bigg(\frac{V(r)}{\sinh^2 r}-\frac{V_{r}(r)\cosh r}{\sinh r}\bigg)|\nabla_{r,\hn} u|^2\dv \,.
	\end{align*}
This completes the proof. \hfill $\Box$

\par \bigskip \par

{\bf Proof of Theorem~\ref{th_Bessel_rellich_non_rad}.}

	Again, by spherical decomposition we can write $u$ as in  \eqref{decomp}. Granting the advantage of $\psi(r)=\sinh r$, we can write some relations like $\frac{{\psi^{\prime}}^2}{\psi^2}=1+\frac{1}{\psi^2}$ and ${\psi^{\prime}}^2=1+\psi^2$ and we will use these identities in the proof frequently. Now, using \eqref{sphr_lap} for the decomposed function $u$, we get
	\begin{align*}
		&\int_{\hn}V(r)|\Delta_{\hn} u|^2\dv =\sum_{n = 0}^{\infty}\bigg[\int_{0}^{\infty} V(r) \left( a_{n}^{\prime \prime} + (N-1) \frac{\psi^{\prime}}{\psi} a_{n}^{\prime} \right)^2 \psi^{N-1}\dr\\&+ \lambda_n^2\int_{0}^{\infty}V(r)\frac{a_{n}^2}{\psi^4}\psi^{N-1}\dr- 2\; \lambda_{n}  \int_{0}^{\infty} V(r) \left( a_{n}^{\prime \prime} +
		(N-1) \frac{\psi^{\prime}}{\psi} a_{n}^{\prime} \right) \frac{a_{n}}{\psi^2} \psi^{N-1} \ {\rm d}r\bigg].\notag
	\end{align*}
	Exploiting Corollary \ref{rop_th_Bessel_rellich} for each $a_n$, we deduce 
	\begin{align*}
		&\int_{\hn}V(r)|\Delta_{\hn} a_n|^2\dv=N\omega_N\int_{0}^{\infty} V(r) \left( a_{n}^{\prime \prime} + (N-1) \frac{\psi^{\prime}}{\psi} a_{n}^{\prime} \right)^2 \psi^{N-1}\dr \\&=N\omega_N\bigg[\int_{0}^{\infty} W{a_n^{\prime}}^2\psi^{N-1}\dr+\int_{0}^{\infty}Vf^2\left[\left(\frac{a_n'}{f}\right)'\right]^2  \psi^{N-1}\dr\\&-(N-1)\int_{0}^{\infty}V\frac{f^\prime}{f}\bigg(\coth r - \frac{1}{r}\bigg)(a_n^{\prime})^2\psi^{N-1}\dr\\&+(N-1)\int_{0}^{\infty}V{a_n^{\prime}}^2\psi^{N-3}\dr-(N-1)\int_{0}^{\infty}V_r\psi^{\prime}(a_n')^2\psi^{N-2}\dr\bigg].
	\end{align*}
	On the other hand, the r.h.s of inequality \eqref{eq_nrad_rellich} in terms of spherical decomposition writes
	\begin{align*}
		&\int_{\hn}W(r)|\nabla_{\hn} u|^2\dv+\int_{\hn}V(r)(f(r))^2\bigg|\nabla_{\hn}\bigg(\frac{u_r}{f(r)}\bigg)\bigg|^2\dv\\&-(N-1)\int_{\hn}V(r)\frac{f^\prime}{f}\bigg(\coth r - \frac{1}{r}\bigg)|\nabla_{\hn} u|^2\dv\notag \\
		& +(N-1)\int_{\hn}\bigg(\frac{V(r)}{\sinh^2 r}-\frac{V_{r}(r)\cosh r}{\sinh r}\bigg)|\nabla_{\hn} u|^2\dv\\& =\sum_{n = 0}^{\infty}\bigg[\int_{0}^{\infty}W{a_n^\prime}^2\psi^{N-1}\dr+\lambda_{n}\int_{0}^{\infty}Wa_n^2\psi^{N-3}\dr+\int_{0}^{\infty}Vf^2\left[\left(\frac{a_n'}{f}\right)'\right]^2\psi^{N-1}\dr\\
		&+\lambda_{n}\int_{0}^{\infty}V{a_n^\prime}^2\psi^{N-3}\dr-(N-1)\int_{0}^{\infty}V\frac{f^\prime}{f}\bigg(\coth r - \frac{1}{r}\bigg){a_n^{\prime}}^2\psi^{N-1}\dr\\
		&-(N-1)\lambda_{n}\int_{0}^{\infty}V\frac{f^\prime}{f}\bigg(\coth r - \frac{1}{r}\bigg)a_n^2\psi^{N-3}\dr\\
		&+(N-1)\int_{0}^{\infty}\bigg(\frac{V(r)}{\psi^2}-\frac{\psi^{\prime}}{\psi}V_{r}(r)\bigg)(a_{n}')^2\psi^{N-1}\dr\\
		&+(N-1)\lambda_{n}\int_{0}^{\infty}\bigg(\frac{V(r)}{\psi^2}-\frac{\psi^{\prime}}{\psi}V_{r}(r)\bigg)\frac{a_{n}^2}{\psi^2}\psi^{N-1}\dr\bigg]\,.
	\end{align*} 
	Therefore, we will be done if we prove that the following quantity $\mathcal{B}$ is non-negative: 
	\begin{align}\label{non_neg_2}
		\mathcal{B}&:= \sum_{n = 0}^{\infty}\bigg[\lambda_n^2\int_{0}^{\infty}V(r)\frac{a_{n}^2}{\psi^4}\psi^{N-1}\dr - 2\; \lambda_{n}  \int_{0}^{\infty}  V(r)\left( a_{n}^{\prime \prime} +
		(N-1) \frac{\psi^{\prime}}{\psi} a_{n}^{\prime} \right) \frac{a_{n}}{\psi^2} \psi^{N-1} \ {\rm d}r\\&-\lambda_{n}\int_{0}^{\infty}W(r)\frac{a_{n}^2}{\psi^2}\psi^{N-1}\dr-(N-1)\lambda_{n}\int_{0}^{\infty}\bigg(\frac{V(r)}{\psi^2}-\frac{\psi^{\prime}}{\psi}V_{r}(r)\bigg)\frac{a_{n}^2}{\psi^2}\psi^{N-1}\dr\notag\\&
		-\lambda_n\int_{0}^\infty V (a_n')^2\psi^{N-3}\dr+(N-1)\lambda_{n}\int_{0}^{\infty}V\frac{f^\prime}{f}\bigg(\coth r - \frac{1}{r}\bigg)a_n^2\psi^{N-3}\dr\notag\bigg].
	\end{align} 
	
	To show that $\mathcal{B}$ is non-negative, we establish some preliminary identities. Let $b_n(r):=\frac{a_n(r)}{\psi(r)}$, by Leibniz rule we have $a_n^{\prime}=b_{n}^{\prime}\psi+b_n\psi^{\prime}$. Using this and by parts formula, we obtain
	\begin{align}\label{nrad_rellich_2}
		\int_{0}^{\infty}V {a_{n}^{\prime}}^2\psi^{N-3}\dr&=\int_{0}^{\infty}V {b_{n}^{\prime}}^2\psi^{N-1}\dr-(N-3)\int_{0}^{\infty}V b_{n}^2\psi^{N-3}\dr\\&-\int_{0}^{\infty}V_rb_n^2\psi^{\prime}\psi^{N-2}\dr-(N-2)\int_{0}^{\infty}Vb_n^2\psi^{N-1}\dr.\notag
	\end{align}
	
	Then applying Lemma~\ref{th_Bessel_hardy} for $b_n$, we deduce
	\begin{align*}
		\int_{0}^{\infty}V {b_{n}^{\prime}}^2\psi^{N-1}\dr&= \int_{0}^{\infty}Wb_n^2\psi^{N-1}\dr+\int_{0}^\infty Vf^2 \left[\bigg(\frac{b_n}{f}\bigg)'\right]^2 \psi^{N-1}\dr\\&-(N-1)\int_{0}^{\infty}V\frac{f^\prime}{f}\bigg(\coth r - \frac{1}{r}\bigg)b_n^2\psi^{N-1}\dr.
	\end{align*}
	
	Using this estimate into \eqref{nrad_rellich_2} and writing $b_n$ in terms of $a_n$, we have
	\begin{align}\label{nrad_rellich_3}
		&\int_{0}^{\infty}V {a_{n}^{\prime}}^2\psi^{N-3}\dr=\int_{0}^{\infty}W a_{n}^2\psi^{N-3}\dr+\int_{0}^\infty Vf^2 \left[\bigg(\frac{a_n}{f\psi}\bigg)'\right]^2\psi^{N-1}\dr\\&-(N-1)\int_{0}^{\infty}V\frac{f^\prime}{f}\bigg(\coth r - \frac{1}{r}\bigg)a_n^2\psi^{N-3}\dr-(N-3)\int_{0}^{\infty}V a_{n}^2\psi^{N-5}\dr\notag\\&-\int_{0}^{\infty}V_ra_n^2\psi^{\prime}\psi^{N-4}\dr-(N-2)\int_{0}^{\infty}Va_n^2\psi^{N-3}\dr.\notag
	\end{align}
	
	\medspace
	
	Before proving $\mathcal{B}$ is non-negative, first exploiting by parts formula, we evaluate separately some terms. First there holds
	\begin{align}\label{nr_rellich_6}
		\int_{0}^{\infty} Va_n^{\prime\prime}a_n\psi^{N-3}\dr=& \frac{1}{2}\int_{0}^{\infty}V_{rr}a_n^2\psi^{N-3}\dr+\frac{(N-3)}{2}\int_{0}^{\infty}V_ra_n^2\psi^{\prime}\psi^{N-4}\dr\\&-\int_{0}^{\infty}V(a_n^{\prime})^2\psi^{N-3}\dr-(N-3)\int_{0}^{\infty}Va_n^{\prime}a_n\psi^{\prime}\psi^{N-4}\dr\notag
	\end{align}
	and then
	\begin{align}\label{nr_rellich_7}
		\int_{0}^{\infty}Va_n^{\prime}a_n\psi^{\prime}\psi^{N-4}\dr &=-\frac{1}{2}\int_{0}^{\infty}V_ra_n^2\psi^{\prime}\psi^{N-4}\dr\\& - \frac{(N-4)}{2}\int_{0}^{\infty}Va_n^2\psi^{N-5}\dr-\frac{(N-3)}{2}\int_{0}^{\infty}Va_n^2\psi^{N-3}\dr.\notag
	\end{align}
	
	\medspace
	
	Next, using \eqref{nrad_rellich_3}, \eqref{nr_rellich_6} and \eqref{nr_rellich_7} into \eqref{non_neg_2} and after delicate calculations, we derive
	\begin{align*}
		& \mathcal{B}  = \sum_{n = 0}^{\infty}\bigg[\lambda_n^2\int_{0}^{\infty}Va_n^2\psi^{N-5}\dr-2\lambda_{n}\int_{0}^{\infty} Va_n^{\prime\prime}a_n\psi^{N-3}\dr\\
		&
		- 2(N-1)\lambda_{n}\int_{0}^{\infty} Va_n^{\prime}a_n\psi^{\prime}\psi^{N-4}\dr-\lambda_{n}\int_{0}^{\infty}Wa_n^2\psi^{N-3}\dr\\
		&-(N-1)\lambda_{n}\int_{0}^{\infty}Va_n^2\psi^{N-5}\dr+(N-1)\lambda_{n}\int_{0}^{\infty}V_ra_n^2\psi^{\prime}\psi^{N-4}\dr\\
		&-\lambda_n\int_{0}^\infty V (a_n')^2\psi^{N-3}\dr+(N-1)\lambda_{n}\int_{0}^{\infty}V\frac{f^\prime}{f}\bigg(\coth r - \frac{1}{r}\bigg)a_n^2\psi^{N-3}\dr\bigg]\\
		& =\sum_{n = 0}^{\infty}\bigg[\lambda_n^2\int_{0}^{\infty}Va_n^2\psi^{N-5}\dr-2\lambda_{n}\biggl\{ \frac{1}{2}\int_{0}^{\infty}V_{rr}a_n^2\psi^{N-3}\dr \\
		& +\frac{(N-3)}{2}\int_{0}^{\infty}V_ra_n^2\psi^{\prime}\psi^{N-4}\dr-\int_{0}^{\infty}V(a_n^{\prime})^2\psi^{N-3}\dr\\
		&-(N-3)\int_{0}^{\infty}Va_n^{\prime}a_n\psi^{\prime}\psi^{N-4}\dr\biggr\}-2(N-1)\lambda_{n}\int_{0}^{\infty}Va_n^{\prime}a_n\psi^{\prime}\psi^{N-4}\dr\\
		&-\lambda_{n}\int_{0}^{\infty}Wa_n^2\psi^{N-3}\dr-(N-1)\lambda_{n}\int_{0}^{\infty}Va_n^2\psi^{N-5}\dr\\
		&+(N-1)\lambda_{n}\int_{0}^{\infty}V_ra_n^2\psi^{\prime}\psi^{N-4}\dr-\lambda_n\int_{0}^\infty V (a_n')^2\psi^{N-3}\dr\\
		&+(N-1)\lambda_{n}\int_{0}^{\infty}V\frac{f^\prime}{f}\bigg(\coth r - \frac{1}{r}\bigg)a_n^2\psi^{N-3}\dr\bigg]\\
		&=\sum_{n = 0}^{\infty}\bigg[\lambda_n^2\int_{0}^{\infty}Va_n^2\psi^{N-5}\dr-2\lambda_{n}\biggl\{ \frac{1}{2}\int_{0}^{\infty}V_{rr}a_n^2\psi^{N-3}\dr-\int_{0}^{\infty}V{a_n^{\prime}}^2\psi^{N-3}\dr\biggr\}\\
		&-4\lambda_{n}\biggl\{-\frac{1}{2}\int_{0}^{\infty}V_ra_n^2\psi^{\prime}\psi^{N-4}\dr- \frac{(N-4)}{2}\int_{0}^{\infty}Va_n^2\psi^{N-5}\dr\\
		&-\frac{(N-3)}{2}\int_{0}^{\infty}Va_n^2\psi^{N-3}\dr\biggr\}-\lambda_{n}\int_{0}^{\infty}Wa_n^2\psi^{N-3}\dr\\
		&-(N-1)\lambda_{n}\int_{0}^{\infty}Va_n^2\psi^{N-5}\dr+2\lambda_{n}\int_{0}^{\infty}V_ra_n^2\psi^{\prime}\psi^{N-4}\dr\\
		&-\lambda_n\int_{0}^\infty V (a_n')^2\psi^{N-3}\dr+(N-1)\lambda_{n}\int_{0}^{\infty}V\frac{f^\prime}{f}\bigg(\coth r - \frac{1}{r}\bigg)a_n^2\psi^{N-3}\dr\bigg]\\
		&=\sum_{n = 0}^{\infty}\bigg[\lambda_{n}(\lambda_{n}+N-7)\int_{0}^{\infty}Va_n^2\psi^{N-5}\dr-\lambda_{n}\int_{0}^{\infty}V_{rr}a_n^2\psi^{N-3}\dr\\
		&+2\lambda_{n}\int_{0}^{\infty}V{a_n^{\prime}}^2\psi^{N-3}\dr+4\lambda_{n}\int_{0}^{\infty}V_ra_n^2\psi^{\prime}\psi^{N-4}\dr\\
		&+2\lambda_{n}(N-3)\int_{0}^{\infty}Va_n^2\psi^{N-3}\dr-\lambda_{n}\int_{0}^{\infty}Wa_n^2\psi^{N-3}\dr\\
		&-\lambda_n\int_{0}^\infty V (a_n')^2\psi^{N-3}\dr+(N-1)\lambda_{n}\int_{0}^{\infty}V\frac{f^\prime}{f}\bigg(\coth r - \frac{1}{r}\bigg)a_n^2\psi^{N-3}\dr\bigg]\\
		&= \sum_{n = 0}^{\infty}\bigg[\lambda_{n}(\lambda_{n}+N-7)\int_{0}^{\infty}Va_n^2\psi^{N-5}\dr-\lambda_{n}\int_{0}^{\infty}V_{rr}a_n^2\psi^{N-3}\dr\\
		&+4\lambda_{n}\int_{0}^{\infty}V_ra_n^2\psi^{\prime}\psi^{N-4}\dr+\lambda_{n}\biggl\{\int_{0}^{\infty}W a_{n}^2\psi^{N-3}\dr\\
		&+\int_{0}^\infty Vf^2\left[\bigg(\frac{a_n}{f\psi}\bigg)'\right]^2\psi^{N-1}\dr-(N-3)\int_{0}^{\infty}V a_{n}^2\psi^{N-5}\dr\\
		&-\int_{0}^{\infty}V_ra_n^2\psi^{\prime}\psi^{N-4}\dr-(N-2)\int_{0}^{\infty}Va_n^2\psi^{N-3}\dr\biggr\}\\
		&+2\lambda_{n}(N-3)\int_{0}^{\infty}Va_n^2\psi^{N-3}\dr-\lambda_{n}\int_{0}^{\infty}Wa_n^2\psi^{N-3}\dr\bigg]\\
		&=\sum_{n = 0}^{\infty}\bigg[\lambda_{n}(\lambda_{n}-4)\int_{0}^{\infty}Va_n^2\psi^{N-5}\dr-\lambda_{n}\int_{0}^{\infty}V_{rr}a_n^2\psi^{N-3}\dr\\
		&+3\lambda_{n}\int_{0}^{\infty}V_ra_n^2\psi^{\prime}\psi^{N-4}\dr+\lambda_{n}(N-4)\int_{0}^{\infty}Va_n^2\psi^{N-3}\dr\\
		&+\lambda_{n}\int_{0}^\infty Vf^2\left[\bigg(\frac{a_n}{f\psi}\bigg)'\right]^2\psi^{N-1}\dr\bigg]\\
		& =\sum_{n = 0}^{\infty}\lambda_{n}\bigg[(\lambda_{n}-4)\int_{0}^{\infty}Va_n^2\psi^{N-5}\dr+\int_{0}^{\infty}\biggl\{\frac{3V_r\psi^{\prime}}{\psi}-V_{rr}\biggr\}a_n^2\psi^{N-3}\dr\\&+(N-4)\int_{0}^{\infty}Va_n^2\psi^{N-3}\dr+\int_{0}^\infty Vf^2\left[\bigg(\frac{a_n}{f\psi}\bigg)'\right]^2\psi^{N-1}\dr\bigg]\\
		&\geq \sum_{n = 0}^{\infty}\lambda_{n}\bigg[\int_{0}^{\infty}\biggl\{(N-5)\frac{V}{\psi^2}+3\frac{V_r\psi^{\prime}}{\psi}-V_{rr}+(N-4)V\biggr\}a_n^2\psi^{N-3}\dr\\
		&+\int_{0}^\infty Vf^2\left[\bigg(\frac{a_n}{f\psi}\bigg)'\right]^2\psi^{N-1}\dr\bigg].
	\end{align*}
	In the last line we have used $\lambda_{n}\geq N-1$ for all $n\geq 1$. Hence, $\mathcal{B}$ eventually turns out to be non-negative due to the hypothesis \eqref{extra_condition_2} and the non negativity of the last term.
\hfill $\Box$
 \medskip


{\bf Proof of Theorem \ref{rellich_usp}.}
The idea of the proof is similar to that of Theorem \ref{HPWhn}. First, exploiting the given conditions into Theorem \ref{th_Bessel_rellich_non_rad}, we deduce that for all $u\in C_c^\infty(\hn\setminus\{x_0\})$ there holds
	\begin{align*}
		\int_{\hn}V(r)|\Delta_{\hn} u|^2\dv\geq \int_{\hn}\tilde W(r)|\nabla_{\hn} u|^2\dv\,.
	\end{align*}
	Finally, we use H\"older inequality and the above to get:
	\begin{align*}
		\bigg(\int_{\hn} |\nabla_{\hn}u|^2\dv\bigg)^2&=\bigg(\int_{\hn}\sqrt{\tilde W(r)}|\nabla_{\hn}u|\frac{|\nabla_{\hn}u|}{\sqrt{\tilde W(r)}}\dv\bigg)^2\\&\leq  \bigg(\int_{\hn}\frac{|\nabla_{\hn}u|^2}{\tilde W(r)}\dv\bigg) \bigg(\int_{\hn}\tilde W(r)|\nabla_{\hn}u|^2\dv\bigg) \\&\leq \bigg(\int_{\hn}\frac{|\nabla_{\hn}u|^2}{\tilde W(r)}\dv\bigg)\bigg(\int_{\hn}V(r)|\Delta_{\hn} u|^2\dv\bigg)\,.
	\end{align*}
The proof of Heisenberg-Pauli-Weyl uncertainty principle involving radial part of the operator is in a similar line.
\hfill $\Box$


\section*{Appendix: a family of improved Hardy-Poincar\'e equalities}
In this appendix we present a family of improved Hardy-Poincar\'e equalities which follows as a corollary from \cite[Theorem 3.2]{FLLM}, i.e. Lemma \ref{th_Bessel_hardy} above, by exploiting the family of Bessel pairs $(r^{N-1}, r^{N-1}W_{\lambda})$ introduced in Section \ref{applications} for all $0 \leq \lambda \leq \lambda_{1}(\hn)$. If $\lambda=\lambda_1(\hn)$ the identity we got is already known from \cite[Theorem 3.2]{FLLM} while for $0\leq\lambda<\lambda_1(\hn)$ it is new and improves \eqref{improved-poinc-lambda}, i.e. \cite[Theorem 2.1]{BGGP}, with the presence of an exact remainder term. The precise statement of the result reads as follows:
\begin{theorem}\label{bessel2}
Let $N \geq 2.$ For all $0 \leq \lambda \leq \lambda_{1}(\hn) = \left(\frac{N-1}{2} \right)^2$  and for all $u \in C_c^{\infty}(\hn \setminus \{ x_0\})$ there holds 
\begin{align*}
&\int_{\hn} |\nabla_{\hn}u|^2 \, {\rm d}v_{\hn} = \lambda \int_{\hn} u^2 \, {\rm d}v_{\hn} + h^2_{N}(\lambda) \int_{\hn} \frac{u^2}{r^2} \, {\rm d}v_{\hn} \notag \\
& + \left[\frac{(N-2)^2}{4} - h^2_{N}(\lambda)  \right] \int_{\hn} \frac{u^2}{\sinh^2 r} \, {\rm d}v_{\hn}+ \gamma_N(\lambda) h_N(\lambda) \int_{\hn} \frac{r \coth r - 1}{r^2} \,u^2 \, {\rm d}v_{\hn}   \notag \\
&+ \int_{\hn}(\Psi_{\lambda}(r))^2\bigg|\nabla_{\hn}\bigg(\frac{u}{\Psi_{\lambda}(r)}\bigg)\bigg|^2\dv
\end{align*}
and for the radial operator we have 
\begin{align*}
&\int_{\hn}  |\nabla_{r,\hn}u|^2 \, {\rm d}v_{\hn} = \lambda \int_{\hn} u^2 \, {\rm d}v_{\hn} + h^2_{N}(\lambda) \int_{\hn} \frac{u^2}{r^2} \, {\rm d}v_{\hn} \notag \\
& + \left[\frac{(N-2)^2}{4} - h^2_{N}(\lambda)  \right] \int_{\hn} \frac{u^2}{\sinh^2 r} \, {\rm d}v_{\hn}+ \gamma_N(\lambda) h_N(\lambda) \int_{\hn} \frac{r \coth r - 1}{r^2} \,u^2 \, {\rm d}v_{\hn}   \notag \\
& + \int_{\hn}(\Psi_{\lambda}(r))^2\bigg|\nabla_{r,\hn}\bigg(\frac{u}{\Psi_{\lambda}(r)}\bigg)\bigg|^2\dv\,,
\end{align*}
where $\gamma_{N}(\lambda):=\sqrt{(N-1)^2-4\lambda}$, $h_N(\lambda):=\frac{\gamma_{N}(\lambda)+1}{2}$ and $\Psi_{\lambda}(r) := r^{-\frac{N-2}{2}} \left(\frac{\sinh r}{r} \right)^{-\frac{N-1+\gamma_N(\lambda)}{2}}$. 
\end{theorem}
\begin{proof}
The proof follows by applying \cite[Theorem~3.2]{FLLM}, i.e. Lemma \ref{th_Bessel_hardy} above, with the Bessel pairs $(r^{N-1}, r^{N-1}W_{\lambda}),$ where $W_{\lambda}$ is as given in \eqref{potential}.
\end{proof}

In particular, for $\lambda = N-2$ Theorem \ref{bessel2} yields the Hardy identity below which improves \eqref{improved-hardy-eq}:

\begin{corollary}
Let $N \geq 3.$ For all $u \in C_{c}^{\infty} (\hn \setminus \{ x_0 \})$ there holds
\begin{align*}\label{appendiximproved-hardy-eq}
\int_{\hn} |\nabla_{\hn} u|^2 \ \emph{d}v_{\hn} & = \left(\frac{N-2}{2} \right)^2 \int_{\hn} \frac{u^2}{r^2} \ \emph{d}v_{\hn} +  (N-2) \int_{\hn} u^2 \ \emph{d}v_{\hn} \notag \\
& + \frac{(N-2)(N-3)}{2} \int_{\hn} \frac{r \coth r - 1}{r^2}  \, u^2 \ \emph{d}v_{\hn}\, \\
& + \int_{\hn} \left(\frac{r^{1/2}}{\sinh r}\right)^{2(N-2)} \left| \nabla_{\hn} \left( \frac{(\sinh r)^{N-2} u}{r^{(N-2)/2}} \right) \right|^2  \, {\rm d}v_{\hn}.
\end{align*}
\end{corollary}

   \par\bigskip\noindent
\textbf{Acknowledgments.}
E.~Berchio is member of the Gruppo Nazionale per l'Analisi Matematica, la Probabilit\`a e le loro Applicazioni (GNAMPA) of the Istituto Nazionale di Alta Matematica (INdAM) and is partially supported by the PRIN project 201758MTR2: ``Direct and inverse problems for partial differential equations: theoretical aspects and applications'' (Italy).
D.~Ganguly is partially supported by the INSPIRE faculty fellowship (IFA17-MA98).  
P.~Roychowdhury is supported by the Council of Scientific \& Industrial Research (File no.  09/936(0182)/2017-EMR-I). Also this work is part of PhD program at Indian Institute of Science Education and Research, Pune. D.~Ganguly is grateful to S.~Mazumdar for useful discussions.

\end{document}